\documentclass[reqno, 10pt]{amsart}
\usepackage[centertags]{amsmath} \usepackage{amsfonts}
\usepackage{amssymb} \usepackage{amsthm}
\usepackage{newlfont}
\renewcommand{\emptyset}{\varnothing}

\newtheorem{thm}{Theorem}[section]

\newtheorem{lemma}[thm]{Lemma}
\newtheorem{prop}[thm]{Proposition}

\newtheorem{example}[thm]{Example}

\numberwithin{equation}{section}

\renewcommand{\footnote}{\endnote}
\newcommand{\ignore}[1]{}\makeglossary

\begin{document}

\title{Rings as the unions of proper subrings}
\author{Andrea Lucchini and Attila Mar\'oti}
\footnotetext[1]{The research of the second author was supported
by a Marie Curie International Reintegration Grant within the 7th
European Community Framework Programme and partially by OTKA
NK72523 and OTKA T049841.} 
\date{19th of January, 2010}
\maketitle

\begin{abstract}
We describe all possible ways how a ring can be expressed as the
union of three of its proper subrings. This is an analogue for
rings of a 1926 theorem of Scorza about groups. We then determine
the minimal number of proper subrings of the simple matrix ring
$M_{n}(q)$ whose union is $M_{n}(q)$.
\end{abstract}

\section{Introduction}

No group is the union of two of its proper subgroups. It is a 1926
theorem of Scorza \cite{S} that a group $G$ is a union of three of
its pairwise distinct proper subgroups $A$, $B$, $C$ if and only
if $A$, $B$, $C$ have index $2$ in $G$ and $G/(A \cap B \cap C)$
is isomorphic to the Klein four group. This result was twice
reproved in \cite{HR} and \cite{BBM}.

No ring is the union of two of its proper subrings, however the
following example of I. Ruzsa \cite{BBBK} shows that a ring can be
the union of three proper subrings. The polynomial ring
$\mathbb{Z}[x]$ is the union of the proper subrings $S_{1}$,
$S_{2}$, $S_{3}$ where $S_{1}$ is the ring consisting of all
polynomials $f$ for which $f(0)$ is even, $S_{2}$ is the ring
consisting of all polynomials $f$ for which $f(1)$ is even, and
$S_{3}$ is the ring consisting of all polynomials $f$ for which
$f(0)+f(1)$ is even. It is easy to see that in this example the
ring $S_{1} \cap S_{2} \cap S_{3}$ is an ideal in $\mathbb{Z}[x]$
and the corresponding factor ring is isomorphic to
$\mathbb{Z}/2\mathbb{Z} \oplus \mathbb{Z}/2\mathbb{Z}$.

Hence it is natural to ask: is there an analogue of Scorza's
result for rings?

Clearly, it is sufficient to classify all ring $R$ and all proper
subrings $S_{1}$, $S_{2}$, $S_{3}$ of $R$ with the property that
$R = S_{1} \cup S_{2} \cup S_{3}$ and that no non-trivial ideal of
$R$ is contained in $S_{1} \cap S_{2} \cap S_{3}$. This leads to
the following definition.

We say that a $4$-tuple $(R,S_{1},S_{2},S_{3})$ of rings is good
if $S_{1}$, $S_{2}$, $S_{3}$ are proper subrings of the ring $R$
so that $R = S_{1} \cup S_{2} \cup S_{3}$ and that no non-trivial
ideal of $R$ is contained in $S_{1} \cap S_{2} \cap S_{3}$. For
any permutation $\pi$ of $\{ 1,2,3 \}$ we consider the good
$4$-tuples $(R,S_{1},S_{2},S_{3})$ and
$(R,S_{1\pi},S_{2\pi},S_{3\pi})$ to be the same. Similarly, if
$\varphi$ is an isomorphism between rings $R$ and $\bar{R}$ and
$(R,S_{1},S_{2},S_{3})$ is a good $4$-tuple, then the $4$-tuples
$(R,S_{1},S_{2},S_{3})$ and
$(\bar{R},\varphi(S_{1}),\varphi(S_{2}),\varphi(S_{3}))$ are also
considered to be the same.

The first result of the paper is

\begin{thm}
\label{t1} All good $4$-tuples of rings (see above) are completely
described by Examples 2.1 - 2.10.
\end{thm}

We say that a ring $R$ is good if there exists a good $4$-tuple of
rings $(R,S_{1},S_{2},S_{3})$. The following result is an analogue
for rings of Scorza's theorem about groups.

\begin{thm}
\label{t2} A ring $R$ is the union of three of its proper subrings
if and only if there exists a factor ring (of order $4$ or $8$) of
$R$ which is isomorphic to a good ring of Example 2.1, 2.2, 2.3,
2.4, or 2.6.
\end{thm}

Note that the setup of a ring expressed as the union of three
proper subrings appeared naturally in the paper \cite{D} of
Deaconescu.

For a ring $R$ that can be expressed as the union of finitely many
proper subrings let $\sigma(R)$ be the minimal number of proper
subrings of $R$ whose union is $R$. In our last theorem we give a
formula for $\sigma(M_{n}(q))$ where $M_{n}(q)$ is the full matrix
ring of $n$-by-$n$ matrices over the field of $q$ elements where
$n \geq 2$.

\begin{thm}
\label{t3} Let $n$ be a positive integer at least $2$. Let $b$ be
the smallest prime divisor of $n$ and let $N(b)$ be the number of
subspaces of an $n$-dimensional vector space over the field of $q$
elements which have dimensions not divisible by $b$ and at most
$n/2$. Then we have
$$\sigma(M_{n}(q)) = \frac{1}{b} \underset{b \nmid
i}{\underset{i=1}{\prod^{n-1}}} (q^{n}-q^{i}) + N(b).$$
\end{thm}

Note that, by Theorem \ref{t3}, $\sigma(M_{2}(2)) = 4$.

Similar investigations to Theorem \ref{t3} for groups have been
carried out in \cite{BEGHM}.

Finally we make an important remark. When trying to determine
$\sigma(R)$ for a given ring $R$ that can be expressed as the
union of finitely many proper subrings, it is sufficient to assume
that $R$ is finite. Indeed, suppose that $k = \sigma(R)$ and
$S_{1}, \ldots , S_{k}$ are proper subrings of $R$ whose union is
$R$. Then, by a result of Neumann \cite{N}, every subring $S_{i}$
($i = 1, \ldots, k$) is of finite index in $R$ (just by
considering the additive structures of all these rings). Hence $S
= S_{1} \cap \ldots \cap S_{k}$ is also a ring of finite index in
$R$. But then, by a result of Lewin \cite{L}, $S$ contains an
ideal $I$ of $R$ of finite index in $R$. Hence $R/I$ is a finite
ring with $\sigma(R/I) = \sigma(R)$.

\section{Examples}

\begin{example}
Let $R$ be the subring $$\left\{ \left(
\begin{matrix}
0  & 0 \\
0  & 0
\end{matrix} \right), \left(
\begin{matrix}
1  & 0 \\
0  & 1
\end{matrix} \right), \left(
\begin{matrix}
1  & 0 \\
0  & 0
\end{matrix} \right), \left(
\begin{matrix}
0  & 0 \\
0  & 1
\end{matrix} \right)  \right\}$$ of $M_{2}(\mathbb{Z}/2\mathbb{Z})$. This
is a commutative ring of order $4$ with a multiplicative identity.
Every non-zero element of $R$ lies inside a unique subring of
order $2$. Hence there are three proper non-zero subrings of $R$.
Let these be $S_{1}$, $S_{2}$, and $S_{3}$. Note that $R$ is
isomorphic to $\mathbb{Z}/2\mathbb{Z} \oplus
\mathbb{Z}/2\mathbb{Z}$.
\end{example}

\begin{example}
Let $R$ be the subring $$\left\{ \left(
\begin{matrix}
0  & 0 & 0 \\
0  & 0 & 0 \\
0  & 0 & 0
\end{matrix} \right), \left(
\begin{matrix}
0  & 0 & 0 \\
1  & 0 & 0 \\
0  & 0 & 0
\end{matrix} \right), \left(
\begin{matrix}
0  & 0 & 0 \\
0  & 0 & 0 \\
1  & 0 & 0
\end{matrix} \right), \left(
\begin{matrix}
0  & 0 & 0 \\
1  & 0 & 0 \\
1  & 0 & 0
\end{matrix} \right)  \right\}$$ of $M_{3}(\mathbb{Z}/2\mathbb{Z})$.
This is a commutative ring of order $4$. It has no multiplicative
identity since it is a zero ring. The ring $R$ has exactly three
subrings of order $2$. Let these be $S_{1}$, $S_{2}$, and $S_{3}$.
Note that $R$ is isomorphic to the subring $$\left\{ \left(
\begin{matrix}
0  & 0 \\
0  & 0
\end{matrix} \right), \left(
\begin{matrix}
2  & 0 \\
0  & 0
\end{matrix} \right), \left(
\begin{matrix}
0  & 0 \\
0  & 2
\end{matrix} \right), \left(
\begin{matrix}
2  & 0 \\
0  & 2
\end{matrix} \right)  \right\}$$ of $M_{2}(\mathbb{Z}/4\mathbb{Z})$.
\end{example}

\begin{example}
Let $R$ be the subring $$\left\{ \left(
\begin{matrix}
0  & 0 \\
0  & 0
\end{matrix} \right), \left(
\begin{matrix}
0  & 1 \\
0  & 0
\end{matrix} \right), \left(
\begin{matrix}
1  & 0 \\
0  & 0
\end{matrix} \right), \left(
\begin{matrix}
1  & 1 \\
0  & 0
\end{matrix} \right)  \right\}$$ of $M_{2}(\mathbb{Z}/2\mathbb{Z})$.
This is a non-commutative ring of order $4$. It has no
multiplicative identity. The ring $R$ has exactly three subrings
of order $2$. Let these be $S_{1}$, $S_{2}$, and $S_{3}$.
\end{example}

\begin{example}
Let $R$ be the subring $$\left\{ \left(
\begin{matrix}
0  & 0 \\
0  & 0
\end{matrix} \right), \left(
\begin{matrix}
0  & 1 \\
0  & 0
\end{matrix} \right), \left(
\begin{matrix}
0  & 0 \\
0  & 1
\end{matrix} \right), \left(
\begin{matrix}
0  & 1 \\
0  & 1
\end{matrix} \right)  \right\}$$ of $M_{2}(\mathbb{Z}/2\mathbb{Z})$.
This is a non-commutative ring of order $4$. It has no
multiplicative identity. The opposite ring $R^{\mathrm{op}}$ of
$R$ is the ring $R$ of Example 2.3. The ring $R$ has exactly three
subrings of order $2$. Let these be $S_{1}$, $S_{2}$, and $S_{3}$.
\end{example}

\begin{example}
Let $R$ be the subring of $M_{3}(\mathbb{Z}/2\mathbb{Z})$
consisting of all matrices of the form
$$\left(
\begin{matrix}
a  & 0 & 0 \\
0  & b & 0 \\
0  & 0 & c
\end{matrix} \right)$$ where $a$, $b$, $c$ are elements
of $\mathbb{Z}/2\mathbb{Z}$. The ring $R$ is isomorphic to
$\mathbb{Z}/2\mathbb{Z} \oplus \mathbb{Z}/2\mathbb{Z} \oplus
\mathbb{Z}/2\mathbb{Z}$. This is a commutative ring of order $8$
with a multiplicative identity. The ring $R$ has three subrings of
order $4$ containing the multiplicative identity of $R$. These are
$S_{1}$ defined by the restriction $a+b=0$, $S_{2}$ defined by the
restriction $a+c=0$, and $S_{3}$ defined by the restriction
$b+c=0$.
\end{example}

\begin{example}
Let $R$ be the subring of $M_{3}(\mathbb{Z}/2\mathbb{Z})$
consisting of all matrices of the form
$$\left(
\begin{matrix}
a  & 0 & 0 \\
b  & a & 0 \\
c  & 0 & a
\end{matrix} \right)$$ where $a$, $b$, $c$ are elements of
$\mathbb{Z}/2\mathbb{Z}$. This is a commutative ring of order $8$.
Note that the subset of $R$ obtained by imposing the restriction
$a=0$ is isomorphic to the ring $R$ of Example 2.2. Indeed $R$ can
be obtained from $R$ of Example 2.2 by adding a multiplicative
identity $1$ and imposing the relation $1+1=0$. The ring $R$ has
three subrings of order $4$ containing the multiplicative identity
of $R$. These are $S_1$ defined by the restriction $b=0$, $S_{2}$
defined by the restriction $c=0$, and $S_{3}$ defined by the
restriction $b+c=0$.
\end{example}

\begin{example}
Let $R$ be the subring of $M_{2}(\mathbb{Z}/2\mathbb{Z})$
consisting of all upper triangular matrices of the form
$$\left(
\begin{matrix}
a  & c \\
0  & b
\end{matrix} \right)$$
where $a$, $b$, $c$ are elements of $\mathbb{Z}/2\mathbb{Z}$. This
is a non-commutative ring of order $8$ containing a multiplicative
identity. The opposite ring $R^{\mathrm{op}}$ of $R$ is the
subring of lower triangular matrices of
$M_{2}(\mathbb{Z}/2\mathbb{Z})$. The rings $R$ and
$R^{\mathrm{op}}$ are isomorphic. The ring $R$ contains exactly
three subrings of order $4$ containing the multiplicative identity
of $R$. These are $S_{1}$ defined by the restriction $c=0$,
$S_{2}$ defined by the restriction $a+b=0$, and $S_{3}$ defined by
the restriction $a+b+c=0$.
\end{example}

\begin{example}
Let $R$ be the subring of $M_{4}(\mathbb{Z}/2\mathbb{Z})$
consisting of all matrices of the form
$$\left(
\begin{matrix}
0  & b & c & d \\
0  & e & 0 & 0 \\
0  & 0 & e & 0 \\
0  & 0 & 0 & e
\end{matrix} \right)$$ where $b$, $c$, $d$, $e$ are elements of
$\mathbb{Z}/2\mathbb{Z}$ subject to the restriction $b+e = 0$.
This is a non-commutative ring of order $8$ without a
multiplicative identity. Let $S_{1}$ be the subring of $R$ defined
by the restriction $c=0$, let $S_{2}$ be the subring of $R$
defined by the restriction $d=0$, and let $S_{3}$ be the subring
of $R$ defined by the restriction $c+d=0$.
\end{example}

\begin{example}
Let $R$ be the subring of $M_{4}(\mathbb{Z}/2\mathbb{Z})$
consisting of all matrices of the form
$$\left(
\begin{matrix}
0  & 0 & 0 & 0 \\
b  & e & 0 & 0 \\
c  & 0 & e & 0 \\
d  & 0 & 0 & e
\end{matrix} \right)$$ where $b$, $c$, $d$, $e$ are elements of
$\mathbb{Z}/2\mathbb{Z}$ subject to the restriction $b+e = 0$.
This is a non-commutative ring of order $8$ without a
multiplicative identity. The ring $R$ is the opposite ring
$R^{\mathrm{op}}$ of the ring $R$ of Example 2.8. Let $S_{1}$ be
the subring of $R$ defined by the restriction $c=0$, let $S_{2}$
be the subring of $R$ defined by the restriction $d=0$, and let
$S_{3}$ be the subring of $R$ defined by the restriction $c+d=0$.
\end{example}

\begin{example}
Let $R$ be the subring of $M_{4}(\mathbb{Z}/2\mathbb{Z})$
consisting of all matrices of the form
$$\left(
\begin{matrix}
a  & 0 & 0 & 0 \\
b  & e & 0 & 0 \\
c  & 0 & e & 0 \\
d  & 0 & 0 & e
\end{matrix} \right)$$ where $a$, $b$, $c$, $d$, $e$ are elements of
$\mathbb{Z}/2\mathbb{Z}$ subject to the restriction $a+b+e = 0$.
This is a non-commutative ring of order $16$ with a multiplicative
identity. The rings $R$ and $R^{\mathrm{op}}$ are isomorphic. The
ring $R$ can be obtained from $R$ of Example 2.8 or $R$ of Example
2.9 by adding a multiplicative identity $1$ and imposing the
relation $1+1=0$. Let $S_{1}$ be the subring of $R$ defined by the
restriction $c=0$, let $S_{2}$ be the subring of $R$ defined by
the restriction $d=0$, and let $S_{3}$ be the subring of $R$
defined by the restriction $c+d=0$.
\end{example}

Let $R$ be a good ring of Examples 2.3 or Example 2.8. Then
$R^{\mathrm{op}}$ is a good ring of Example 2.4 or Example 2.9
respectively. The rings $R$ and $R^{\mathrm{op}}$ are not
isomorphic since their left and right annihilators have different
sizes.

\section{Reductions}

Let $(R,S_{1},S_{2},S_{3})$ be a good $4$-tuple of rings. Then $R$
is a good ring. Let $S = S_{1} \cap S_{2} \cap S_{3}$. Note that
by Scorza's theorem, each $(S_{i},+)$ ($i \in \{ 1,2,3 \}$) has
index $2$ in $(R,+)$ and $S_{1} \cap S_{2} = S_{1} \cap S_{3} =
S_{2} \cap S_{3} = S$.

Let $2R$ denote the set of all elements of $R$ of the form $r+r$
for $r \in R$. It is easy to see that $2R$ is an ideal of $R$.
Moreover, by Scorza's theorem, the abelian groups $S_{1}$,
$S_{2}$, $S_{3}$ all have index $2$ in $R$, hence $r+r \in S_{i}$
for all $r \in R$ and $i \in \{ 1,2,3 \}$. Thus $2R$ is an ideal
of $R$ contained in $S$. This forces $2R = 0$.

Since $2R = 0$, we may assume that there exists elements $x$ and
$y$ of $R$ such that
$$(R,+)=S\oplus \{x,y,x+y,0\},$$
$$(S_1,+)=S\oplus \{x,0\},\quad (S_2,+)=S\oplus \{y,0\},\quad (S_3,+)=S\oplus \{x+y,0\}.$$

\begin{lemma}
\label{l1} For any $s\in S$ we have $sx \in S \Leftrightarrow sy
\in S$ and $xs \in S \Leftrightarrow ys \in S$.
\end{lemma}

\begin{proof}
Assume, for example, for a contradiction, that $sx \in S$ and $sy
\not\in S$. Then there exists $s_1,s_2\in S$ with $sx=s_1$ and
$sy=s_2+y$. This implies $s(x+y)=s_1+s_2+y \not\in S_3$ against
the fact that $S_3$ is a subring.
\end{proof}

Define $S_R:=\{s \in S\mid sx \in S\}$, $S_L:=\{s \in S\mid xs \in
S\}$, and $T:=S_L\cap S_R$. Notice that $S_R$ and $S_T$ are
subgroups of $(S,+)$ with index at most $2$, so $T$ is a subgroup
of $(S,+)$ with $|S:T|$ equal to $1$, $2$, or $4$. Moreover, by
Lemma \ref{l1}, if $t \in T$ then $\{tx,xt,ty,yt\}\subseteq S$.

\begin{lemma}
If $t \in T$ then $xty \in S$, $ytx\in S$, $xtx\in S$, and $yty\in
S$.
\end{lemma}

\begin{proof} Assume that $t \in T$. Since $xt \in S$, we must
have $xty \in S_2$ and since $ty \in S$, we must have $xty \in
S_1$. Hence $xty \in S_1\cap S_2=S$. The same argument works for
$ytx$. Notice that $xt \in T$ implies also that $xt(x+y)\in S_3$;
moreover we have that $xtx=s_1+bx$, $xty=s_2$ with $s_1,s_2\in S$,
$b\in \{0,1\}$. We must have $xt(x+y)=s_1+s_2+bx \in S_3$, hence
$b=0$. The same argument works for $yty$.
\end{proof}

Now assume that $T \neq \{ 0 \}$ and take $0 \neq t \in T$. We
have $RtR\subseteq S$. Indeed for any $r_1$, $r_2\in R$ we have
$r_1=s_1+a_1x+b_1y$ and $r_2=s_2+a_2x+b_2y$ for some $s_{1}$,
$s_{2} \in S$ and $a_{1}$, $a_{2}$, $b_{1}$, $b_{2} \in \{ 0, 1
\}$. Hence $r_{1}tr_{2}$ is equal to
$$s_1ts_2+a_2s_1tx+b_2s_1ty+a_1xts_2+a_1a_2xtx+a_1b_2xty+b_1yts_2+b_1a_2ytx+b_1b_2yty.$$
We would have that $RtR$ is a non-trivial ideal of $R$ contained
in $S$, a contradiction. (We may assume that $RtR$ is non-trivial.
There are three possibilities. If $Rt \not= \{ 0 \}$, then $Rt$ is
a non-trivial ideal of $R$ contained in $S$, a contradiction. If
$tR \not= \{ 0 \}$, then $tR$ is a non-trivial ideal of $R$
contained in $S$, a contradiction. Finally, if $Rt = tR = \{ 0
\}$, then the abelian group generated by $t$ is an ideal of $R$
contained in $S$, a contradiction.) This means that $T = \{ 0 \}$
and this implies that $|S|=|S:T|$ is $1$, $2$, or $4$.

We proved the following reduction.

\begin{prop}
Let $R$ be a good ring. Then $|R| = 4$, $8$, or $16$.
\end{prop}

Suppose that $M$ is a ring with or without a multiplicative
identity. Then consider the abelian group $M^{*} = M \oplus
\langle u \rangle$ with $u + u = 0$. Now define a multiplication
on $M^{*}$ by setting $u$ to be the identity on $M^{*}$ and
extending the product according to the distributive laws. Thus
$M^{*}$ becomes a ring with a multiplicative identity.

\begin{prop}
\label{p3} Let $(R,S_{1},S_{2},S_{3})$ be a good $4$-tuple of
rings. Suppose that $R$ has no multiplicative identity. Then
$(R^{*},S_{1}^{*},S_{2}^{*},S_{3}^{*})$ is also a good $4$-tuple
of rings where a unique multiplicative identity was added to the
four rings $R$, $S_{1}$, $S_{2}$, and $S_{3}$.
\end{prop}

\begin{proof}
Clearly $R^{*} = S_{1}^{*} \cup S_{2}^{*} \cup S_{3}^{*}$ is again
a union of proper subrings. Assume that $J$ is an ideal of $R^{*}$
with $J \subseteq S^{*} := S_{1}^{*} \cap S_{2}^{*} \cap
S_{3}^{*}$. Notice that $I := J \cap S = J \cap R$ is an ideal of
$R$ contained in $S$ hence $I = \{ 0 \}$. Since $|R^{*}:R| = 2$,
this implies $|J| \leq 2$. Assume, by contradiction, that $J \not=
\{ 0 \}$. Then there exists $u \not= r \in R$ such that $J =
\langle u - r \rangle$. Notice that $r \not= 0$. Then for any
non-zero $z \in R$ we have $z(u-r) = z - zr$ and $(u-r)z = z-rz$.
Both these expressions are in $R \cap J = \{ 0 \}$, hence $z = rz
= zr$, in other words $r$ behaves as a multiplicative identity in
$R$. This is a contradiction.
\end{proof}

\begin{prop}
\label{p1} Let $(R,S_{1},S_{2},S_{3})$ be a good $4$-tuple of
rings. Suppose that $R$ contains a multiplicative identity, $1$.
Then either $R \cong \mathbb{Z}/2\mathbb{Z} \oplus
\mathbb{Z}/2\mathbb{Z}$ or $1 \in S_{i}$ for all $i$ with $1 \leq
i \leq 3$.
\end{prop}

\begin{proof}
If $i$ is an index with $1 \not\in S_{i}$ then $S_{i}$ is an ideal
in $R$ (since $S_{i}$ has index $2$ in $R$ and $(1+s_{1})s_{2} \in
S_{i}$ and $s_{2}(1+s_{1}) \in S_{i}$ for all $s_{1}$, $s_{2} \in
S_{i}$). Suppose that there exist indices $i \not= j$ with $1
\not\in S_{i}$ and $1 \not\in S_{j}$. Then $S_{i} \cap S_{j} =
S_{1} \cap S_{2} \cap S_{3}$ is an ideal in $R$. Hence $S_{i} \cap
S_{j} = \{ 0 \}$ and so $R \leq R/S_{i} \oplus R/S_{j}$. This
forces $|R|=4$ and $R = \langle 1,x \rangle$. Moreover $\langle
1+x \rangle$ must be a subring and so ${(1+x)}^{2} = 1+x^{2} \in
\langle 1+x \rangle$ hence $x = x^{2}$. Thus $R = \langle 1,x
\rangle \cong \mathbb{Z}/2\mathbb{Z} \oplus
\mathbb{Z}/2\mathbb{Z}$. So we may assume, without loss of
generality, that $1 \in S_{1} \cap S_{2} = S_{1} \cap S_{2} \cap
S_{3}$ which finishes the proof of the proposition.
\end{proof}

\section{Rings with multiplicative identity}

In this section we will classify all good $4$-tuples of rings
$(R,S_{1},S_{2},S_{3})$ where $R$ is a ring with a multiplicative
identity.

\begin{prop}
Let $(R,S_{1},S_{2},S_{3})$ be a good $4$-tuple of rings. Suppose
that $R$ has a multiplicative identity and that $|R|=4$. Then
$(R,S_{1},S_{2},S_{3})$ is of Example 2.1.
\end{prop}

\begin{proof}
We may assume that $R = \{ 0,1,a,1+a \}$, that $1+1 = 0$, that
$a^{2}=0$ or $1$, and that ${(1+a)}^{2} = 0$ or $1+a$. The latter
two conditions force $a^{2}=a$. This implies the result.
\end{proof}

\begin{prop}
Let $(R,S_{1},S_{2},S_{3})$ be a good $4$-tuple of rings. Suppose
that $R$ has a multiplicative identity and that $|R|=8$. Suppose
that the Jacobson radical $J(R)$ of $R$ is trivial. Then
$(R,S_{1},S_{2},S_{3})$ is of Example 2.5.
\end{prop}

\begin{proof}
Since $|R|=8$ and $J(R)= \{ 0 \}$, by the Artin-Wedderburn
theorem, there are three possibilities for $R$. The ring $R$ can
be isomorphic to $GF(8)$, to $GF(4) \oplus GF(2)$, or to $GF(2)
\oplus GF(2) \oplus GF(2)$. In the first case no proper subring of
$R$ contains the primitive elements of $R$. Suppose that the
second case holds. Let $a$ be a generator of the multiplicative
group of $GF(4)$. Then the element $(a,1)$ must be contained in a
proper subring of $R = GF(4) \oplus GF(2)$, say in $S_{1}$. But
then $S_{1}$ cannot be a ring of order $4$ since $(1,1)$ is also
contained in $S_{1}$ by Proposition \ref{p1}. This is a
contradiction. Hence only the third case can hold. But the third
case can indeed hold as shown by Example 2.5.
\end{proof}

We continue with two easy lemmas.

\begin{lemma}
\label{l2} Let $R$ be a good ring of order $8$. Suppose that $R$
has a multiplicative identity. Then for any $r \in R$ different
from $0$ or $1$, the elements $0$, $1$, $r$, $1+r$ form a subring
of $R$.
\end{lemma}

\begin{proof}
Let $r$ be an arbitrary element of $R$ different from $0$ or $1$.
Since $R$ is a good ring, there exists a subring $S_{1}$ of order
$4$ containing $r$. By Proposition \ref{p1}, we know that $1$ is
also contained in $S_{1}$. Hence $S_{1} = \{ 0,1,r,1+r \}$.
\end{proof}

\begin{lemma}
\label{l3} Let $R$ be a good ring of order $8$ with a
multiplicative identity. Then $u^{2}=0$ for every element $u$ of
$J(R)$.
\end{lemma}

\begin{proof}
We may assume that $u \not= 0$, $1$. Then, by Lemma \ref{l2}, the
elements $0$, $1$, $u$, and $1+u$ form a subring of $R$. Hence
$u^{2}$ is either $0$, $1$, $u$, or $1+u$.

Note that since $u$ is in $J(R)$ the elements $1+zu$ and $1+uz$
are invertible in $R$ for every element $z$ of $R$.

Suppose that $u^{2} = 1$. Then ${(1+u)}^{2}=1+u^{2} = 0$
contradicting the fact that $1+u$ is invertible. Suppose that
$u^{2} = 1+u$. Then $1+u^{2} = u$ is invertible which would mean
that $J(R) = R$, a contradiction. Suppose that $u^{2} = u$. Then
$(1+u)u = u + u^{2} = 0$ contradicting the fact that $1+u$ is
invertible.
\end{proof}

We are now in the position to show Proposition \ref{p2}.

\begin{prop}
\label{p2} Let $(R,S_{1},S_{2},S_{3})$ be a good $4$-tuple of
rings. Suppose that $R$ has a multiplicative identity and that
$|R|=8$. Suppose that $|J(R)|=2$. Then $(R,S_{1},S_{2},S_{3})$ is
of Example 2.7.
\end{prop}

\begin{proof}
We may assume that $R$ consists of the $8$ elements $0$, $1$, $x$,
$1+x$, $y$, $1+y$, $x+y$, $1+x+y$. Without loss of generality,
assume that $J(R) = \{ 0,y \}$. Then $y^{2}=0$ by Lemma \ref{l3}.
By Lemma \ref{l2}, we know that $x^{2} = a + bx$ for some $a$, $b
\in \{ 0,1 \}$. Similarly, since $y \in J(R)$ and $J(R)$ is an
ideal of $R$, we have $xy = cy$ and $yx = dy$ for some $c$, $d \in
\{ 0,1 \}$. Now, again by Lemma \ref{l2}, we have ${(x+y)}^{2} =
x^{2} + y^{2} + xy + yx = a + bx + (c+d)y \in \{ 0, 1, x+y, 1+x+y
\}$. Hence $b = c+d$.

Suppose for a contradiction that $b=0$. Then $x^{2}=a$. Without
loss of generality, we may assume that $a=0$, for otherwise
${(x+1)}^{2} = 0$ and hence we could replace $x$ by $1+x$. Since
$c+d = b = 0$, the ring $R$ is commutative. Hence for any $r \in
R$ we have ${(1+rx)}^{2} = 1 + {(rx)}^{2} = 1$. This means that
$1+rx$ is invertible and so $x \in J(R)$. This is a contradiction.

We conclude that $b=1$. There are hence two possibilities for $c$
and $d$. From these two possibilities we get that in $R$ we either
have $xy = y$ and $yx = 0$, or $xy = 0$ and $yx = y$. In either
case it can be shown that $x^{2}=x$. Since the two arguments in
the two cases are similar, we only give the proof in the first
case. From  $x^{2} = a+x$ we see that $0 = (yx)x = y{x}^{2} = y
(a+x) = ay + yx = ay$ from which we conclude that $a=0$.

There are hence two possibilities for the good ring $R$ of order
$8$. These two possibilities give rise to opposite rings.

Let us consider the first possibility for $R$. In this case $R$ is
defined by the relations $y^{2}=0$, $xy = y$, $yx = 0$, and
$x^{2}=x$. Identifying $x$ with the matrix $\left(
\begin{matrix}
1  & 0 \\
0  & 0
\end{matrix} \right)$
and $y$ with the matrix $\left(
\begin{matrix}
0  & 1 \\
0  & 0
\end{matrix} \right)$,
we see that $R$ is isomorphic to the ring of upper triangular
matrices in $M_{2}(\mathbb{Z}/2\mathbb{Z})$. The opposite ring of
$R$ is isomorphic to the ring of lower triangular matrices in
$M_{2}(\mathbb{Z}/2\mathbb{Z})$ which in fact is isomorphic to
$R$.

By Proposition \ref{p1}, we know that $1 \in S_{i}$ for all $i$
with $i \in \{ 1,2,3 \}$. We also know that the $S_{i}$'s must
have order $4$. Hence there is essentially one possibility for the
$S_{i}$'s. This proves that a good $4$-tuple
$(R,S_{1},S_{2},S_{3})$ exists and it is of Example 2.7.
\end{proof}

\begin{prop}
\label{p5} Let $(R,S_{1},S_{2},S_{3})$ be a good $4$-tuple of
rings. Suppose that $R$ has a multiplicative identity and that
$|R|=8$. Suppose that $|J(R)| = 4$. Then $(R,S_{1},S_{2},S_{3})$
is of Example 2.6.
\end{prop}

\begin{proof}
As before, we may assume that $R$ consists of the $8$ elements
$0$, $1$, $x$, $1+x$, $y$, $1+y$, $x+y$, $1+x+y$. Without loss of
generality, assume that $J(R) = \{ 0,x,y,x+y \}$.

By Lemma \ref{l3}, we have that $x^{2} = y^{2} = {(x+y)}^{2} = 0$.
Hence $0 = {(x+y)}^{2} = xy + yx$ implies $xy = yx$. Now $xy = ax
+ by$ for some $a$, $b \in \{ 0,1 \}$ since $J(R)$ is an ideal.
Hence $0 = x^{2}y = x(ax+by) = bxy$ and $0 = xy^{2} = (ax+by)y =
axy$. Thus $a = b = 0$ and so $xy = yx = 0$.

Such a ring $R$ exists. By identifying $x$ with the matrix $\left(
\begin{matrix}
0  & 0 & 0 \\
1  & 0 & 0 \\
0  & 0 & 0
\end{matrix} \right)$ and $y$ with the matrix $\left(
\begin{matrix}
0  & 0 & 0 \\
0  & 0 & 0 \\
1  & 0 & 0
\end{matrix} \right)$ we see that $R$ is isomorphic to the ring
$R$ of Example 2.6.

By Proposition \ref{p1}, we know that $1 \in S_{i}$ for all $i$
with $i \in \{ 1,2,3 \}$. We also know that the $S_{i}$'s must
have order $4$. Hence there is essentially one possibility for the
$S_{i}$'s. This proves that a good $4$-tuple
$(R,S_{1},S_{2},S_{3})$ exists and it is of Example 2.6.
\end{proof}

\begin{prop}
\label{p4} Let $(R,S_{1},S_{2},S_{3})$ be a good $4$-tuple of
rings. Suppose that $|R|=16$. Then $(R,S_{1},S_{2},S_{3})$ is of
Example 2.10.
\end{prop}

\begin{proof}
Let $S = S_{1} \cap S_{2} \cap S_{3}$. By the beginning of Section
3, we know that $(R,+) = S \oplus \{ 0,x,y,x+y \}$ for some
elements $x$ and $y$. By Proposition \ref{p1}, $1 \in S$. Recall
the definitions of $S_{R}$, $S_{L}$, and $T = S_{R} \cap S_{L}$
from Section 3. From the proofs in Section 3 it is clear that
$|S_{R}| = |S_{L}| = 2$ since $|S|=4$. It is also clear that
$|T|=0$. From this we see that there exists a unique $a \in S$
with $ax \in S$ and $xa \not\in S$. (It is clear that $a$ is
different from $0$ and $1$ and that $S = \{ 0,1,a,1+a \}$.)

We claim that we may assume that $x^{2} \in S$. If $x^{2} \not\in
S$ then $x^{2} = s + x$ for some $s \in S$. (This follows from the
fact that $x$ and $x^{2}$ must lie inside the subring of $R$, say
$S_{1}$ of order $8$, generated (as an abelian group) by $(S,+)$
and $x$.) In this case ${(x+a)}^{2} = x^{2} + a^{2} + ax + xa = (s
+ a^{2} + ax) + x(1+a)$ where both summands are inside $S$. (The
second summand is in $S$ since $S_{L} = \{ 0,1+a \}$.) Hence there
is no harm to substitute $x$ with $x+a$.

Next we claim that $a^{2}=a$. Notice that $a^{2}x = a(ax) \in S$
hence $a^{2} \in S_{R} = \{ 0,a \}$. Write $xa$ in the form $s +
x$ for some $s \in S$. (This can be done as explained in the
previous paragraph.) Then $$xa^{2} = (xa)a = (s+x)a = sa + xa = sa
+ s + x \not\in S.$$ This implies that $a^{2} \not= 0$.

Now we claim that $xa = x+s$ with $s \in \{ 0, 1+a \}$. Indeed,
$$x+s = xa = xa^{2} = (xa)a = (x+s)a = xa + sa = x+s+sa$$ implies
$sa=0$ which in turn implies the claim.

We claim that we may assume that $xa=x$. Indeed, if $xa = x+a+1$
then $(x+1)a = x+1+a+a = x+1$. Moreover ${(x+1)}^{2} = x^{2} +1
\in S$. Hence in this case we may substitute $x$ with $x+1$.

We claim that $ax \in \{ 0,a \}$. Indeed, $a(ax) = a^{2}x = ax$
and $ax \in S = \{ 0,1,a,1+a \}$. It can be checked that $ax \not=
1$ or $1+a$.

We claim that $ax = 0$. Let us assume for a contradiction that $ax
= a$. Then $x^{2} = (xa)x = x(ax) = xa = x$. But $x^{2} \in S$ and
$x \not\in S$ is a contradiction.

It follows that $x^{2}=0$ since $x^{2} = (xa)x = x(ax) = 0$.

Analogues of the above claims can be stated and proved for $y$
instead of $x$. Hence, to summarize what we have obtained, we have
the relations $x^{2} = y^{2} = 0$, $xa = x$, $a^{2} = a$, $ax = ay
= 0$, and $ya = y$.

We claim that $xy \in S$ and $yx \in S$. We will only prove that
$xy \in S$. The argument for $yx \in S$ is similar. We start with
the observation that $y^{2}=0$ implies $(xy)y = 0$. Assume that
$xy$ has the form $s + \alpha x + \beta y$ for some $s \in S$ and
$\alpha$, $\beta$ from $\{ 0,1 \}$. By the previous observation we
have $$0 = (s + \alpha x + \beta y)y = sy + \alpha xy = sy +
\alpha s + \alpha^{2}x + \alpha \beta y$$ from which it follows
that $\alpha = 0$. We continue with the observation that $x^{2}=0$
implies $x(xy)=0$. Then $0 = x(s + \beta y) = xs + \beta (s +
\beta y) = xs + \beta s + {\beta}^{2}y$ which implies $\beta = 0$.
This proves the claim.

Finally, we claim that $xy = yx = 0$. We will only show that $xy =
0$ since the proof of the claim that $yx = 0$ is similar. By the
previous claim, we know that $xy \in \{ 0,1,a,1+a \}$. Now $x^{2}y
= x(xy) = 0$ implies that $xy \in \{ 0, 1+a \}$. But $xy = 1+a$
would mean that $0 = (xy)y = (1+a)y = y$. A contradiction.

A unique ring $R$ exists with the derived restrictions on the
multiplications. The above proof also shows that $R$ is isomorphic
to $R^{\mathrm{op}}$.

Our ring $R$ is isomorphic to the ring $R$ of Example 2.10. To see
this it is sufficient to consider the map which sends $x$, $y$,
$a$ to the respective matrices
$$\left(
\begin{matrix}
0  & 0 & 0 & 0 \\
0  & 0 & 0 & 0 \\
1  & 0 & 0 & 0 \\
0  & 0 & 0 & 0
\end{matrix} \right), \quad \left(
\begin{matrix}
0  & 0 & 0 & 0 \\
0  & 0 & 0 & 0 \\
0  & 0 & 0 & 0 \\
1  & 0 & 0 & 0
\end{matrix} \right), \quad \left(
\begin{matrix}
1  & 0 & 0 & 0 \\
1  & 0 & 0 & 0 \\
0  & 0 & 0 & 0 \\
0  & 0 & 0 & 0
\end{matrix} \right).$$

It remains to show that there exists exactly one good $4$-tuple of
rings \\ $(R,S_{1},S_{2},S_{3})$ with $R$ as above. (From Example
2.10 and also from the proof above, it is clear that there exists
at least one good $4$-tuple of rings.) Since $|R|=16$ and $S =
S_{1} \cap S_{2} \cap S_{3}$ has order $4$, a good $4$-tuple of
rings $(R,S_{1},S_{2},S_{3})$ is completely determined by the
$\mathrm{Aut}(R)$-automorphism class of $S$. Hence it is
sufficient to show that if $(R,S_{1},S_{2},S_{3})$ and
$(R,S_{1}',S_{2}',S_{3}')$ are two good $4$-tuples of rings (with
$R$ as above) then there exists an automorphism $\varphi$ of $R$
such that $S_{1}^{\varphi} \cap S_{2}^{\varphi} \cap
S_{3}^{\varphi} = S_{1}' \cap S_{2}' \cap S_{3}'$. Put $S := S_{1}
\cap S_{2} \cap S_{3} = \{ 0,1,a,1+a \}$ and suppose that $S' :=
S_{1}' \cap S_{2}' \cap S_{3}' = \{ 0,1,r,1+r \}$ for some $r \in
R$. Without loss of generality, we may assume that $r = \epsilon a
+ u$ for some $\epsilon \in \{ 0,1 \}$ and some $u \in \{
0,x,y,x+y \}$. But $\epsilon$ cannot be $0$ since otherwise $\{
0,r \}$ would be an ideal of $R$ inside $S'$. Thus $r = a + u$.
But then the map sending the elements $0$, $1$, $a$, $x$, $y$ to
the elements $0$, $1$, $a+u$, $x$, $y$ respectively can naturally
be extended to an automorphism $\varphi$ of $R$ sending $S$ to
$S'$.
\end{proof}

\section{Rings with no multiplicative identity}

In the previous section we classified all good rings with a
multiplicative identity and in this section we will use this
classification to list all good rings without a multiplicative
identity. Our key tool in this project is Proposition \ref{p3}.

In the first case we cannot have a good ring $R$ without a
multiplicative identity such that $R^{*}$ is a good ring of order
$8$.

\begin{prop}
Suppose that $R'$ is the good ring $\mathbb{Z}/2\mathbb{Z} \oplus
\mathbb{Z}/2\mathbb{Z} \oplus \mathbb{Z}/2\mathbb{Z}$. Then if $R$
is a ring with $R^{*} = R'$ then $R \cong \mathbb{Z}/2\mathbb{Z}
\oplus \mathbb{Z}/2\mathbb{Z}$.
\end{prop}

\begin{proof}
The ring $R$ must contain exactly one element of each of the
following sets: $\{ (1,0,0), (0,1,1) \}$, $\{ (0,1,0), (1,0,1)
\}$, $\{ (0,0,1), (1,1,0) \}$. The ring $R$ can only contain one
vector with two $1$'s. Moreover it must contain exactly one such
vector $v$. Without loss of generality $v$ contains a $0$ in the
first entry. Hence the ring $R$ will be the ring consisting of all
vectors with a $0$ in the first entry.
\end{proof}

In the next case we find two good rings.

\begin{prop}
Let $R'$ be the subring of $M_{2}(\mathbb{Z}/2\mathbb{Z})$
consisting of all upper triangular matrices. There exists two
non-commutative rings $R_{1}$ and $R_{2}$ of order $4$ without a
multiplicative identity such that $R_{1}^{*} = R_{2}^{*} = R'$.
These are of Examples 2.3 and 2.4.
\end{prop}

\begin{proof}
A good subring of $R'$ of order $4$ not containing a
multiplicative identity must contain the zero matrix and exactly
one element of each of the following sets of matrices: $$\left\{
\left(
\begin{matrix}
1  & 1 \\
0  & 1
\end{matrix} \right), \left(
\begin{matrix}
0  & 1 \\
0  & 0
\end{matrix} \right) \right\}, \left\{ \left(
\begin{matrix}
1  & 1 \\
0  & 0
\end{matrix} \right), \left(
\begin{matrix}
0  & 1 \\
0  & 1
\end{matrix} \right) \right\}, \left\{ \left(
\begin{matrix}
0  & 0 \\
0  & 1
\end{matrix} \right), \left(
\begin{matrix}
1  & 0 \\
0  & 0
\end{matrix} \right) \right\}.$$
The square of the matrix $\left(
\begin{matrix}
1  & 1 \\
0  & 1
\end{matrix} \right)$ is the identity, so this matrix cannot lie
inside our good ring without a multiplicative identity. So a
possible good ring must contain the matrix $\left(
\begin{matrix}
0  & 1 \\
0  & 0
\end{matrix} \right)$. We are hence left with two possibilities
and these lead us to Examples 2.3 and 2.4.
\end{proof}

The statement of the following proposition is a bit technical but
its proof is short.

\begin{prop}
Let $(R',S_{1}',S_{2}',S_{3}')$ be a good $4$-tuple of rings with
$R'$ a ring of order $8$ with a multiplicative identity. Suppose
also that $|J(R')|=4$. Then there exist a unique good $4$-tuples
of rings $(R,S_{1},S_{2},S_{3})$ with
$(R^{*},S_{1}^{*},S_{2}^{*},S_{3}^{*}) =
(R',S_{1}',S_{2}',S_{3}')$ (where a unique identity is added to
all four rings $R$, $S_{1}$, $S_{2}$, and $S_{3}$). This tuple is
of Example 2.2.
\end{prop}

\begin{proof}
Since $|R'|=8$ it is sufficient to show that there is a unique
good ring $R$ which in fact is a zero ring (since it would be of
order $4$). By Proposition \ref{p5}, we may assume that $R'$ is
generated by the elements $1$, $x$, $y$ subject to the relations
$x^{2}=y^{2}=xy=yx=0$. Since ${(1+x)}^{2} = {(1+y)}^{2} = 1$, the
elements $1+x$ and $1+y$ cannot lie in $R$. Hence $R = \{ 0, x, y,
x+y \}$. This is a zero ring.
\end{proof}

Finally, we consider the good ring of order $16$.

\begin{prop}
\label{p6} Let $(R',S_{1}',S_{2}',S_{3}')$ be a good $4$-tuple of
rings with $|R'|=16$. Then there exist two good $4$-tuples of
rings $(R,S_{1},S_{2},S_{3})$ with
$(R^{*},S_{1}^{*},S_{2}^{*},S_{3}^{*}) =
(R',S_{1}',S_{2}',S_{3}')$ (where a unique identity is added to
all four rings $R$, $S_{1}$, $S_{2}$, and $S_{3}$). One such tuple
is of Example 2.8 and the other is of Example 2.9.
\end{prop}

\begin{proof}
We use the notations of Proposition \ref{p4}. Let $R'$ be the ring
generated by the elements $1$, $a$, $x$, and $y$ subject to the
relations $1+1=0$, $x^{2}=y^{2}= 0$, $ax = ay = xy = yx = 0$,
$xa=x$, $a^{2}=a$, and $ya = y$. We wish to construct rings
$R_{1}$ and $R_{2}$ with $R_{1}^{*} = R_{2}^{*} = R'$. To do this
we need to pick exactly one element from each set $\{ 1+r, r \}$
where $r \in R'$. Since ${(1+x)}^{2} = {(1+y)}^{2} = 1$, the
elements $x$ and $y$ must lie inside $R_{1}$ and $R_{2}$. Let
$R_{1}$ be the ring generated by the elements $a$, $x$, $y$ and
let $R_{2}$ be the ring generated by the elements $1+a$, $x$, $y$.
It is easy to see that $R_{2}$ is isomorphic to $R$ of Example
2.9. It is also clear that $R_{1}$ is the opposite ring of
$R_{2}$. Hence $R_{1}$ is isomorphic to $R$ of Example 2.8. We
noted at the end of Section 2 that the $R$'s of Examples 2.8 and
2.9 are not isomorphic. To finish the proof of the proposition it
is sufficient to show that there is a unique good $4$-tuple of
rings $(R_{1},S_{1},S_{2},S_{3})$. But this follows by the
argument given at the end of the proof of Proposition \ref{p4}. We
just note that $S = S_{1} \cap S_{2} \cap S_{3}$ must have the
form $\{ 0, a+u \}$ for some element $u$ in the ideal of $R_{1}$
generated by $x$ and $y$, and note also that the map sending $x$,
$y$, $a$ to $x$, $y$, $a+u$ respectively can be extended to an
automorphism of $R_{1}$.
\end{proof}

This proves Theorem \ref{t1}.

\section{Proof of Theorem \ref{t2}}

We break the proof of Theorem \ref{t2} up into a series of
propositions. (It is easy to see that it suffices to prove only
these propositions.)

\begin{prop}
The good ring of Example 2.5 has a factor ring isomorphic to the
good ring of Example 2.1
\end{prop}

\begin{proof}
Let $R$ be the good ring of Example 2.5. Then the set $$\left\{
\left(
\begin{matrix}
0  & 0 & 0 \\
0  & 0 & 0 \\
0  & 0 & 0
\end{matrix} \right), \left(
\begin{matrix}
1  & 0 & 0 \\
0  & 0 & 0 \\
0  & 0 & 0
\end{matrix} \right) \right\}$$
is an ideal $I$ of $R$ such that $R/I$ is isomorphic to the good
ring of Example 2.1.
\end{proof}

\begin{prop}
The good ring of Example 2.7 has a factor ring isomorphic to the
good ring of Example 2.1.
\end{prop}

\begin{proof}
Let $R$ be the good ring of Example 2.7. Then the set $$\left\{
\left(
\begin{matrix}
0  & 0  \\
0  & 0
\end{matrix} \right), \left(
\begin{matrix}
0  & 1  \\
0  & 0
\end{matrix} \right) \right\}$$
is an ideal $I$ of $R$ such that $R/I$ is isomorphic to the good
ring of Example 2.1.
\end{proof}

\begin{prop}
The good ring of Examples 2.8 has a factor ring isomorphic to the
good ring of Example 2.4.
\end{prop}

\begin{proof}
The good ring of Example 2.8 is isomorphic to the good ring
$R_{1}$ introduced in the proof of Proposition \ref{p6}. The ring
$R_{1}$ is generated by the elements $x$, $y$, $a$ subject to the
relations $r+r = 0$ for all $r \in R_{1}$, $x^{2} = y^{2} = ax =
ay = xy = yx = 0$, $xa=x$, $a^{2}=a$, and $ya=y$. There is an
ideal $I = \{ 0,y \}$ in $R_{1}$. Then $x$ and $a$ are different
coset representatives in the factor ring $R_{1}/I$. The map
sending $x$ and $a$ to the matrices $$\left(
\begin{matrix}
0  & 1 \\
0  & 0
\end{matrix} \right), \quad
\left(
\begin{matrix}
0  & 0 \\
0  & 1
\end{matrix} \right)$$ respectively extends naturally to an isomorphism
between $R_{1}/I$ and the good ring of Example 2.4.
\end{proof}

\begin{prop}
The good ring of Example 2.9 has a factor ring isomorphic to the
good ring of Examples 2.3.
\end{prop}

\begin{proof}
The good ring of Example 2.9 is isomorphic to the good ring
$R_{2}$ introduced in the proof of Proposition \ref{p6}. The ring
$R_{2}$ is generated by the elements $x$, $y$, $1+a$ subject to
the relations $r+r = 0$ for all $r \in R_{2}$, $x^{2} = y^{2} =
x(1+a) = y(1+a) = xy = yx = 0$, $(1+a)x=x$, $a^{2}=a$, and
$(1+a)y=y$. There is an ideal $I = \{ 0,y \}$ in $R_{2}$. Then $x$
and $1+a$ are different coset representatives in the factor ring
$R_{2}/I$. The map sending $x$ and $1+a$ to the matrices $$\left(
\begin{matrix}
0  & 1 \\
0  & 0
\end{matrix} \right), \quad
\left(
\begin{matrix}
1  & 0 \\
0  & 0
\end{matrix} \right)$$ respectively extends naturally to an isomorphism
between $R_{2}/I$ and the good ring of Example 2.3.
\end{proof}

\begin{prop}
The good ring of Example 2.10 has a factor ring isomorphic to the
good ring of Example 2.1.
\end{prop}

\begin{proof}
Let $R$ be the good ring of Example 2.10. Then the ideal $I$ (of
order $4$) of $R$ generated by the matrices $$\left(
\begin{matrix}
0  & 0 & 0 & 0 \\
0  & 0 & 0 & 0 \\
1  & 0 & 0 & 0 \\
0  & 0 & 0 & 0
\end{matrix} \right) \quad \mathrm{and} \quad \left(
\begin{matrix}
0  & 0 & 0 & 0 \\
0  & 0 & 0 & 0 \\
0  & 0 & 0 & 0 \\
1  & 0 & 0 & 0
\end{matrix} \right)$$ have the property that $R/I$ is isomorphic
to the good ring of Example 2.1.
\end{proof}

The last proposition is not needed for the proof of Theorem
\ref{t2}, however, for the sake of completeness, we include it
here.

\begin{prop}
The good ring of Example 2.6 has no good proper factor ring.
\end{prop}

\begin{proof}
The good ring of Example 2.6 is isomorphic to the ring $R$
generated by the elements $1$, $x$, $y$ subject to the relations
$x^{2}=y^{2}=xy=yx=0$. Suppose, for a contradiction, that $I$ is a
non-trivial ideal of $R$ such that $R/I$ is a good ring. Then
$|I|=2$. Moreover, since ${(1+x)}^{2} = {(1+y)}^{2} =
{(1+x+y)}^{2} = 1$, we have $I = \{ 0,u \}$ for some element $u
\in \{ x, y, x+y \}$. Let $a$ be such that $\langle x,y \rangle =
\langle u \rangle \oplus \langle a \rangle$. Then $R/I \cong
\langle 1,a \rangle$. But the ring $\langle 1,a \rangle$ is
generated by a single element, $1+a$, so it cannot be good. A
contradiction.
\end{proof}

\section{Proof of Theorem \ref{t3}}

In this section we will prove Theorem \ref{t3}.

Let $V$ be the natural module for the ring $M_{n}(q)$ where $n
\geq 2$ and $q$ is a prime power. For a non-trivial proper
subspace $U$ of $V$ let $M(U)$ be the subring of $M_{n}(q)$
consisting of all elements of $M_{n}(q)$ which leave $U$
invariant. For a positive integer $a$ dividing $n$ the ring
$M_{n/a}(q^{a})$ can be embedded in $M_{n}(q)$ in a natural way.
Hence we consider $M_{n/a}(q^{a})$ as a subring of $M_{n}(q)$.
Note that every $GL(n,q)$-conjugate of $M_{n/a}(q^{a})$ is again a
subring of $M_{n}(q)$.

\begin{lemma}
\label{l7} Let $n \geq 2$. Then the maximal subrings of $M_{n}(q)$
are the $M(U)$'s for all non-trivial proper subspaces $U$ of $V$
and the $GL(n,q)$-conjugates of the ring $M_{n/a}(q^{a})$ where
$a$ is a prime divisor of $n$.
\end{lemma}

\begin{proof}
Let $R$ be a subring of $M_{n}(q)$. If $R$ leaves a non-trivial
proper subspace $U$ of $V$ invariant, then $R \subseteq M(U)$.
Hence we may assume that $V$ is an irreducible $R$-module. Let $C$
be the centralizer of $R$ in $M_{n}(q)$. It is clear that $C$ is a
ring. By a variation of Schur's lemma we see that $C$ is a finite
division ring. Thus, by Wedderburn's theorem, $C$ is a finite
field of order $q^{r}$, say. By the double centralizer theorem, we
know that $R = \mathrm{End}_{C}(V)$ and that $R$ is a
$GL(n,q)$-conjugate of $M_{n/r}(q^{r})$. Let $a$ be a prime
divisor of $r$. Then there exists a subfield $D$ of $C$ of order
$q^{a}$. But then $R \subseteq \mathrm{End}_{D}(V)$. This proves
that the listed subrings in the statement of the lemma are the
only possibilities for maximal subrings of $M_{n}(q)$. From the
previous argument it also follows (just by considering centralizer
sizes) that the $GL(n,q)$-conjugates of the ring $M_{n/a}(q^{a})$
are indeed maximal for every prime divisor $a$ of $n$. It is also
easy to see that the subring $M(U)$ is maximal for every
non-trivial proper subspace $U$ of $V$.
\end{proof}

\begin{lemma}
\label{l4} The number of $GL(n,q)$-conjugates of the ring
$M_{n/a}(q^{a})$ is
\\ $|GL(n,q)|/|GL(n/a,q^{a}).a|$.
\end{lemma}

\begin{proof}
Put $X = M_{n/a}(q^{a})$. Let $N$ be the normalizer of $X$ in
$GL(n,q)$ and $C$ be the centralizer of $X$ in $GL(n,q)$. It is
clear that $GL(n/a,q^{a})$ is contained in $N$. The Frobenius
automorphism of order $a$ of the field of order $q^{a}$ is also
contained in $N$. Hence the group $GL(n/a,q^{a}).a$ is contained
in $N$. On the other hand, $N/C$ is a subgroup of the full
automorphism group of $X$, which, by a result of Skolem and
Noether (see Theorem 3.62 of Page 69 of \cite{CR}), has order
equal to $|GL(n/a,q^{a}).a|/(q^{a}-1)$. Hence $|N| =
|GL(n/a,q^{a}).a|$ and the result follows.
\end{proof}

Let $b$ be the smallest prime divisor of $n$ and let $N(b)$ be the
number of subspaces of $V$ which have dimensions not divisible by
$b$ and at most $n/2$.

\begin{prop}
\label{p7} Let $n \geq 2$. Then we have
$$\sigma(M_{n}(q)) \leq \frac{1}{b} \underset{b \nmid
i}{\underset{i=1}{\prod^{n-1}}} (q^{n}-q^{i}) + N(b).$$
\end{prop}

\begin{proof}
Let $\mathcal{H}$ be the set of all $GL(n,q)$-conjugates of
$M_{n/b}(q^{b})$ together with all subrings $M(U)$ where $U$ is a
subspace of $V$ of dimension not divisible by $b$ and at most
$n/2$. By Lemma \ref{l4}, it is sufficient to show that every
element $x$ of $M_{n}(q)$ is contained in a member of
$\mathcal{H}$.

Let $f$ be the characteristic polynomial of $x$. If $f$ is
irreducible, then, by Schur's lemma and Wedderburn's theorem, $x$
is contained in some conjugate of $M_{n/b}(q^{b})$. So we may
assume that $f$ is not an irreducible polynomial.

If $f$ has an irreducible factor of degree $k$, then, by the
theorem on rational canonical forms, $x$ must leave a
$k$-dimensional subspace invariant. So if $k$ is not divisible by
$b$ and at most $n/2$, then $x$ is an element of some member of
$\mathcal{H}$. Hence we may assume that the degree of each
irreducible factor of $f$ is divisible by $b$.

Put $f = f_{1}^{m_{1}} \ldots f_{\ell}^{m_{\ell}}$ where each
$f_{i}$ is a sign times an irreducible polynomial of degree
$r_{i}b$ for some positive integer $r_{i}$. Then, by the theorem
on rational canonical forms, $V = \oplus_{i=1}^{\ell} V_{i}$
viewed as an $\langle x \rangle$-module where for each $i$ the
linear transformation $x$ has characteristic polynomial
${f_{i}}^{m_{i}}$ on the module $V_{i}$. Now each module $V_{i}$
contains an irreducible submodule of dimension $r_{i}b$, and so by
Schur's lemma and Wedderburn's theorem, the centralizer of $x$
contains a field of order $q^{r_{i}b}$, and hence a field of order
$q^{b}$. This means that we may view $x$ as a linear
transformation on $V$ viewed as an $n/b$-dimensional space over a
field of $q^b$ elements, and so $x$ is an element of a
$GL(n,q)$-conjugate of $M_{n/b}(q^{b})$.
\end{proof}

A Singer cycle in $GL(n,q)$ is a cyclic subgroup of order
$q^{n}-1$. It permutes the non-zero vectors of $V$ in one single
cycle. A Singer cycle generates a field of order $q^{n}$ in
$M_{n}(q)$. All Singer cycles in $GL(n,q)$ are $GL(n,q)$-conjugate
to the group $GL(1,q^{n})$ which is a subgroup of $GL(n/a,q^{a})$
for every divisor $a$ of $n$. The normalizer of a Singer cycle is
conjugate to a subgroup of the form $GL(1,q^{n}).n$. The group
$GL(1,q^{n}).n$ lies inside $GL(n/a,q^{a}).a$ for every divisor
$a$ of $n$. The ring $M_{n/a}(q^{a})$ contains exactly
$|GL(n/a,q^{a}).a|/|GL(1,q^{n}).n|$ Singer cycles for every prime
divisor $a$ of $n$. By this and by Lemma \ref{l4} it follows that
every Singer cycle lies inside a unique $GL(n,q)$-conjugate of
$M_{n/a}(q^{a})$ for every divisor $a$ of $n$. Since a Singer
cycle $S$ acts irreducibly on $V$, no ring $M(U)$ contains $S$
where $U$ is a non-trivial proper subspace of $V$. There are
$\varphi(q^{n}-1)$ generators of a Singer cycle where $\varphi$ is
Euler's function.

Let $\Pi_{1}$ be the set of all generators of all Singer cycles on
$V$. Let us call a generator of a Singer cycle an element of type
$T_{0}$.

For every positive integer $k$ with $1 \leq k < n/2$ establish a
bijection $\varphi_{k}$ from the set $\mathcal{S}_{k}$ of all
$k$-dimensional subspaces of $V$ to the set $\mathcal{S}_{n-k}$ of
all $n-k$-dimensional subspaces of $V$ in such a way that for
every $k$-dimensional subspace $U$ we have $V = U \oplus
U\varphi_{k}$. For an arbitrary positive integer $k$ with $1 \leq
k < n/2$ and $b \nmid k$, and for an arbitrary vector space $U \in
\mathcal{S}_{k}$ an element of the form
$$\left(
\begin{matrix}
S_{U}  & 0 \\
0  & S_{U\varphi_{k}}
\end{matrix} \right)$$
where $S_{U}$ is a generator of a Singer cycle on $U$ and
$S_{U\varphi_{k}}$ is a generator of a Singer cycle on $U
\varphi_{k}$ is called an element of type $T_{k}$.

In this paragraph let $n$ be congruent to $2$ modulo $4$. An
element $g$ of $GL(n,q)$ is said to be of type $T_{n/2}$ if there
exist complementary subspaces $U$ and $U'$ of dimensions $n/2$
such that $g$ has the form
$$\left(
\begin{matrix}
S_{U}  & I \\
0  & S_{U'}
\end{matrix} \right)$$
where $I$ is the $n/2$-by-$n/2$ identity matrix and $S_{U}$,
$S_{U'}$ denote the same generator of a Singer cycle acting on $U$
and $U'$ respectively.

Let the set of all elements of type $T_{k}$ for all $k$ (with $1
\leq k < n/2$ and $b \nmid k$) be $\Pi_{2}$ and the set of all
elements of type $T_{n/2}$ be $\Pi_{3}$. Note that if $n$ is not
congruent to $2$ modulo $4$ then $\Pi_{3} = \emptyset$.

\begin{lemma}
\label{l5} Let $k$ be a positive integer with $1 \leq k < n/2$ and
$b \nmid k$. If $R$ is a maximal subring of $M_{n}(q)$ containing
an element of type $T_{k}$, then $R = M(U)$, $M(W)$, or
$g^{-1}M_{n/a}(q^{a})g$ where $U$ is a $k$-dimensional subspace of
$V$, $W$ is an $n-k$-dimensional subspace of $V$, $a$ is any
divisor of $k$, and $g$ is some element of $GL(n,q)$.
\end{lemma}

\begin{proof}
By the proof of Proposition \ref{p7}, it is sufficient to show
that if $a$ is not a divisor of $k$, then the group
$GL(n/a,q^{a})$ contains no element of type $T_{k}$. Suppose for a
contradiction that there exists an element $x$ of type $T_{k}$ in
$GL(n/a,q^{a})$ where $a$ does not divide $k$. Let $C$ be the
centralizer of $x$ in $GL(n,q)$. The size of $C$ is
$(q^{k}-1)(q^{n-k}-1)$. On the other hand, the group of scalars
matrices in $GL(n/a,q^{a})$ is contained in $C$ hence $q^{a}-1$
must divide $(q^{k}-1)(q^{n-k}-1)$. We will show that this is not
the case. In doing so we may assume that $a$ is prime (otherwise
we may take a prime divisor of $a$ to be $a$). It can be shown by
an elementary argument that $$(q^{a}-1,q^{k}-1) = q-1 =
(q^{a}-1,q^{n-k}-1).$$ Hence $q^{a}-1$ must divide ${(q-1)}^{2}$
which is impossible since $q^{a}-1 > {(q-1)}^{2}$.
\end{proof}

\begin{lemma}
\label{l6} Let $n$ be congruent to $2$ modulo $4$. If $R$ is a
maximal subring of $M_{n}(q)$ containing an element of type
$T_{n/2}$, then $R = M(U)$, or $g^{-1}M_{n/a}(q^{a})g$ where $U$
is a $n/2$-dimensional subspace of $V$, $a$ is any divisor of
$n/2$, and $g$ is some element of $GL(n,q)$.
\end{lemma}

\begin{proof}
By the proof of Proposition \ref{p7}, it is sufficient to show
that if $a$ is not a divisor of $n/2$, then the group
$GL(n/a,q^{a})$ contains no element of type $T_{n/2}$. Suppose for
a contradiction that there exists an element $x$ of type $T_{n/2}$
in $GL(n/a,q^{a})$ where $a$ does not divide $n/2$. Then there
exists an element of order $q^{a}-1$ centralizing $x$. Let $c$ be
an arbitrary element centralizing $x$. Then since $x$ leaves a
unique non-trivial proper subspace $U$ of $V$ invariant (which has
dimension $n/2$) it easily follows that $c$ leaves $U$ invariant.
Hence, writing $c$ in block matrix form, we have
$$c = \left(
\begin{matrix}
A  & B \\
0  & C
\end{matrix} \right)$$
for some $n/2$-by-$n/2$ matrices $A$, $B$, and $C$. It is easy to
see that $A$ and $B$ centralize the corresponding generators of
Singer cycles in the block matrix form of $x$. This means that $A$
and $B$ are powers of generators of Singer cycles. In particular,
$A^{q^{n/2}-1} = C^{q^{n/2}-1} = 1$. This means that the order of
$c$ is of the form $\beta p^{\gamma}$ for some positive integer
$\beta$ dividing $q^{n/2}-1$ and for some non-negative integer
$\gamma$ where $p$ denotes the prime divisor of $q$. In
particular, if $c$ is the element of order $q^{a}-1$, then
$q^{a}-1 = \beta \mid q^{n/2}-1$ which is a contradiction since $a
\nmid n/2$.
\end{proof}

Let $\Pi$ be a subset of $M_{n}(q)$. We define $\sigma(\Pi)$ to be
the minimal number of proper subring of $M_{n}(q)$ whose union
contains $\Pi$. Clearly, $\sigma(\Pi) \leq \sigma(M_{n}(q))$. Let
$\mathcal{H} \subseteq \mathcal{K}$ be two sets of subrings of
$M_{n}(q)$. We say that $\mathcal{H}$ is \emph{definitely
unbeatable} on $\Pi$ with respect to $\mathcal{K}$ if the
following four conditions hold.
\begin{enumerate}
\item $\Pi \subseteq \bigcup_{H \in \mathcal{H}} H$;

\item $\Pi \cap H \not= \emptyset$ for all $H \in \mathcal{H}$;

\item $\Pi \cap H_{1} \cap H_{2} = \emptyset$ for all distinct
$H_{1}$ and $H_{2}$ in $\mathcal{H}$; and

\item $|\Pi \cap K| \leq |\Pi \cap H|$ for all $H \in \mathcal{H}$
and all $K \in \mathcal{K} \setminus \mathcal{H}$.
\end{enumerate}

Let $\Pi = \Pi_{1} \cup \Pi_{2} \cup \Pi_{3}$. We aim to determine
$\sigma(\Pi)$. At present there is an important point to make. Let
$\mathcal{C}$ be a set of subrings of $M_{n}(q)$ with the property
that the union of its members contain $\Pi$. Suppose also that
$|\mathcal{C}| = \sigma(\Pi)$. Then we may assume that all members
of $\mathcal{C}$ are maximal subrings of $M_{n}(q)$ and that no
member of $\mathcal{C}$ is $M(W)$ for any subspace $W$ of $V$ of
dimension larger than $n/2$. (The latter statement follows from
the fact that if $W$ is a subspace of dimension $n-k > n/2$ then
$\Pi \cap M(W) = \Pi \cap M(W \varphi_{k}^{-1})$.)

Let $\mathcal{H}$ be the set of maximal subrings of $M_{n}(q)$
consisting of all $GL(n,q)$-conjugates of $M_{n/b}(q^{b})$ and all
subrings of the form $M(U)$ where $U$ is a $k$-dimensional
subspace of $V$ with $b \nmid k$. Let $\mathcal{K}$ be the set of
all maximal subrings of $M_{n}(q)$ apart from the ones which are
of the form $M(W)$ where $W$ is a subspace of $V$ of dimension
larger than $n/2$.

We claim that $\mathcal{H}$ is definitely unbeatable on $\Pi$ with
respect to $\mathcal{K}$. Once we verified this claim we are
finished with the proof of Theorem \ref{t3}. Indeed, the claim
implies that $\sigma(\Pi) = |\mathcal{H}|$. Furthermore, by
Proposition \ref{p7}, we have
$$\frac{1}{b} \underset{b \nmid
i}{\underset{i=1}{\prod^{n-1}}} (q^{n}-q^{i}) + N(b) =
|\mathcal{H}| = \sigma(\Pi) \leq \sigma(M_{n}(q)) \leq \frac{1}{b}
\underset{b \nmid i}{\underset{i=1}{\prod^{n-1}}} (q^{n}-q^{i}) +
N(b).$$

Part (1) of the definition of definite unbeatability follows from
the proof of Proposition \ref{p7}. Let $R \in \mathcal{H}$. If $R
= g^{-1}M_{n/b}(q^{b})g$ for some $g \in GL(n,q)$, then $R$
contains an element of type $T_{0}$ (and no elements of other
types). If $R = M(U)$ for some $l$-dimensional subspace $U$ of
$V$, then $R$ contains an element of type $T_{l}$ (and no elements
of other types). This proves that part (2) of the definition of
definite unbeatability holds. Part (3) follows from our
construction of elements of types $T_{0}$, $T_{k}$, and $T_{n/2}$
and our choice of $\mathcal{H}$. (See the description of Singer
cycles, Lemma \ref{l5}, and Lemma \ref{l6}.) Hence it is
sufficient to show that part (4) of the definition of definite
unbeatability holds.

Let $n$ be a prime power (a power of $b$). Then, by Lemma
\ref{l7}, $\mathcal{K} \setminus \mathcal{H}$ consists of all
subrings of the form $M(U)$ where $U$ is a $k$-dimensional
subspace of $V$ with $b \mid k$ and $k \leq n/2$. Hence, by Lemma
\ref{l5} and Lemma \ref{l6}, we have $|\Pi \cap K| = 0 < |\Pi \cap
H|$ for all $H \in \mathcal{H}$ and all $K \in \mathcal{K}
\setminus \mathcal{H}$. This means that it is sufficient to assume
that $n$ is not a prime power.

Let $c$ be the second largest prime divisor of $n$ (after $b$). It
is clear that \\ $\max \{ |\Pi \cap K| \} \leq |GL(n/c,q^{c})|$
where the maximum is over all $K$ in $\mathcal{K} \setminus
\mathcal{H}$. Hence it is sufficient to show that $|GL(n/c,q^{c})|
\leq |\Pi \cap H|$ for all $H \in \mathcal{H}$. We will next
consider this inequality for the various possibilities of $H$ in
$\mathcal{H}$.

Let $H$ be a ring that is $GL(n,q)$-conjugate to $M_{n/b}(q^{b})$.
Then we have
$$|GL(n/c,q^{c})| < \frac{|GL(n/b,q^{b}).b|}{|GL(1,q^{n}).n|} \varphi(q^{n}-1)
= |\Pi_{1} \cap H| = |\Pi \cap H|.$$

Let $H \in \mathcal{H}$ be a ring of the form $M(U)$ where $U$ is
a $k$-dimensional subspace with $k < n/2$. Then we have
$$|GL(n/c,q^{c})| < \frac{|GL(k,q)|}{|GL(1,q^{k}).k|} \cdot
\frac{|GL(n-k,q)|}{|GL(1,q^{n-k}).(n-k)|} \varphi(q^{k}-1)
\varphi(q^{n-k}-1) =$$ $$= |\Pi_{2} \cap H| = |\Pi \cap H|.$$

Finally, let $H \in \mathcal{H}$ be a ring of the form $M(U)$
where $U$ is a subspace of $V$ of dimension $n/2$. (This is the
case only when $n$ is congruent to $2$ modulo $4$.) Then we have
$$|GL(n/c,q^{c})| < q^{n^{2}/4} \frac{|GL(n/2,q)|}{|GL(1,q^{n/2}).(n/2)|}
\varphi(q^{n/2}-1) = |\Pi_{3} \cap H| = |\Pi \cap H|.$$

The previous three inequalities were derived using the following
three facts. For any positive integer $m$ and prime power $r$ we
have $(1/(m+1))r^{m^{2}} \leq |GL(m,r)|$. (This follows from the
inequality $k/(k+1) \leq 1 - (1/r^{k})$ holding for every positive
integer $k$ between $1$ and $m$.) Secondly, Lemma 5.1 of
\cite{BEGHM} was invoked. Finally, the sequence
$\sqrt[n]{((n/2)+1)(n/2)}$ is monotone decreasing on the set of
even integers whenever $n \geq 6$.

This finishes the proof of Theorem \ref{t3}.

\bigskip

{\it Andrea Lucchini, Dipartimento di Matematica Pura ed
Applicata, Via Trieste

63, 35121 Padova, Italy. E-mail address: lucchini@math.unipd.it}

\medskip

{\it Attila Mar\'oti, MTA Alfr\'ed R\'enyi Institute of
Mathematics, Budapest, Hun-

gary. E-mail address: maroti@renyi.hu}

\end{document}